\documentclass[reqno, 11pt]{amsart}
\usepackage[margin = 1.2 in]{geometry}
\usepackage{amsthm,amsmath,amsfonts}
\usepackage{hyperref}
\usepackage{amssymb}
\usepackage{stmaryrd}
\usepackage{eucal, mathrsfs}
\usepackage[libertine]{newtxmath}
\usepackage{todonotes}
\usepackage{tikz-cd}
\usepackage{listings}
\usepackage{xcolor}

\usepackage[nameinlink]{cleveref}

\theoremstyle{plain}
\newtheorem{theorem}{Theorem}[section]
\newtheorem{proposition}[theorem]{Proposition}
\newtheorem{lemma}[theorem]{Lemma}
\newtheorem{corollary}[theorem]{Corollary}
\newtheorem{definition}[theorem]{Definition}
\newtheorem{question}[theorem]{Question}

\theoremstyle{remark}
\newtheorem{remark}[theorem]{Remark}

\renewcommand{\P}{\mathbb{P}}

\renewcommand{\k}{\mathscr{k}}

\DeclareMathOperator{\I}{\mathcal{I}}

\renewcommand{\O}{\mathscr{O}}
\renewcommand{\L}{\mathscr{L}}
\renewcommand{\S}{\mathscr{S}}

\DeclareMathOperator{\Sym}{Sym}

\DeclareMathOperator{\Z}{\mathbb{Z}}

\DeclareMathOperator{\Gr}{Gr}
\DeclareMathOperator{\CH}{CH}
\DeclareMathOperator{\G}{\mathbb{G}}

\DeclareMathOperator{\ord}{ord}
\DeclareMathOperator{\GL}{GL}
\DeclareMathOperator{\SL}{SL}

\newcommand{\transvectant}{\mathsf{Tv}}

\DeclareMathOperator{\mult}{mult}

\DeclareMathOperator{\red}{red}
\DeclareMathOperator{\init}{init}
\DeclareMathOperator{\Proj}{Proj}

\author{Anand Patel}
\title{Counting sums of two powers}

\begin{document}

\maketitle

\section{Introduction}
\label{section:introduction}

This note is about the enumerative geometry of the locus of binary forms expressible as $f^{a} + g^{b}$.  We aim to generalize two previously known enumerative facts: 
\begin{enumerate}
    \item[(I)] (Clebsch \cite{clebsch1870theorie}) A general binary sextic form is expressible as the sum of a cube and a square in $40$ essentially different ways.
    \item[(II)] (Zariski \cite{zariski1928hyperelliptic}, then Vakil \cite{vakil2001twelve}) If $p$ and $q$ are general binary forms of degree $12$, then there are exactly $3762$ values of $t$ for which $p + t q$ is the sum of a cube and a square.
\end{enumerate}

To state the general question precisely, suppose $\P^{d}$ denotes the projective space of degree $d$ binary forms, and suppose $a$ and $b$ are integers dividing $d$. Then the locus $A$ (resp. $B$) representing binary forms which are $a$'th powers (resp. $b$'th powers) is a smooth projective variety in $\P^{d}$ of dimension $m := d/a$ (resp. $n := d/b$).  The {\sl problem of  counting sums of two powers} is then to determine the degree of the {\sl join} of $A$ and $B$.  Unfortunately, $A$ and $B$ intersect, and their scheme-theoretic intersection is not friendly (e.g. it is non-reduced with embedded primes), and so the degree of the join cannot be calculated by applying basic formulas.  While we will demonstrate (in characteristic zero) how to count sums of two powers in many more cases beyond those of Clebsch, Zariski, and Vakil, we fall very short of a complete solution.  

The difficulty in the problem is governed by the dimension of the intersection $A \cap B$.  A key observation we make here is that the rational map $A \times B \dashrightarrow \G(1,d)$ sending a pair of points to the line they span is resolved by a single variety which only depends on the integers $m$ and $n$, and not on $a, b$ or $d$ individually.  This variety is denoted $\transvectant_{m,n}$ in the text (\Cref{definition:transvectantvariety}), and we call it the {\sl variety of first transvectants} because it is basically the resolution of indeterminacy of the first transvectant $\{p,q\} := p_{x}q_{y} - p_{y}q_{x}$ of two binary forms, this pairing viewed as a rational map $\P^{m} \times \P^{n} \dashrightarrow \P^{m+n-2}$.  

$\transvectant_{m,n}$ is the main subject of our investigation throughout the paper, with its importance to counting sums of two problems made precise in \Cref{proposition:clebsch-resolved}.  We try to obtain an understanding of its enumerative geometry by localizing with respect to a natural torus action.  This analysis reveals how the problem of counting sums of two powers ``reduces'' to the computation of an intersection product in the cohomology ring of $\P^{m} \times \P^{n} \times \P^{m+n-2}$ (\Cref{corollary:solutioninprinciple}).  We successfully follow through with the strategy in the cases where $\gcd(m,n) = 1$ and $\gcd(m,n) = 2$.  (This greatest common divisor is the dimension of $A \cap B$ from above.)  The motivating examples of Clebsch and Zariski-Vakil are in the $\gcd(m,n) = 1$ and $\gcd(m,n) = 2$ categories, respectively.  For another category $1$ example, we find that the locus of binary forms of type $f^{5} + g^{3}$ in $\P^{15}$ has degree $29822$.  And for another category $2$ example, the locus of binary forms of type $f^{5} + g^{2}$ in $\P^{20}$ has degree $626327$. 

In the last section, after mentioning a few interesting questions, we provide a Python script which counts sums of two powers in the first two $\gcd(m,n)$ categories.



\section{Preliminaries} 
\label{section:preliminaries}

    \subsection{Our conventions and notation} We work over an algebraically closed field $\k$ of characteristic $0$, and all schemes are of finite type over $\k$. The term {\sl variety} will mean a reduced, irreducible scheme of finite type over $\k$. If $W$ is a $\k$-vector space, then its projectivization $\P W$ parametrizes lines in $W$ rather than hyperplanes.  
    
    We let $S_j \subset \k[x,y]$ denote the $\k$-vector space of homogeneous polynomials of degree $j$.  Hence $\P S_j$ is the projective space representing degree $d$ binary forms up to scaling.  If $f \in S_j$ is a nonzero form, then we will abuse notation and simply write $f \in \P S_j$ for the corresponding point.

    We let $\Gr(2,S_j)$ denote the Grassmannian representing pencils of binary forms of degree $j$. The spaces $S_{j}$, $\P S_{j}$, $\Gr(2,S_{j})$, etc... all admit natural $\GL(2)$-actions, inherited by linearly substituting the variables $x$ and $y$.

    If $Y$ is a smooth projective variety then $\CH(Y)$ denotes the Chow ring of $Y$, and if $\xi \in \CH(Y)$ is a $0$-cycle, then $\int_{Y} \xi \in \Z$ denotes the degree of $\xi$.
    
    \subsection{Extending a rational map}
    
    The following elementary lemma provides us our first step.  We are grateful to John Doyle for conversations which led to this lemma.
    \begin{lemma}
        \label{lemma:extension}
        Suppose $X, Y$ and $Z$ are varieties with $X$ normal, and with $Y$ assumed proper. Suppose $\alpha : X \dashrightarrow Y$ is a rational map, and that $\beta : Y \to Z$ is a quasi-finite morphism.  Then $\alpha$ is regular if and only if $\beta \circ \alpha$ is regular.
    \end{lemma}
    
    \begin{proof}
        One direction is obvious, so we focus on the non-trivial direction and assume that $\beta \circ \alpha$ is regular. Let $\Gamma_{\alpha} \subset X \times Y$ denote the closure of the graph of $\alpha$, and let $\Gamma_{\beta \circ \alpha} \subset X \times Z$ be the graph of the morphism $\beta \circ \alpha$.  We must show that the first projection $p_1 : \Gamma_{\alpha} \to X$ is an isomorphism.  The projection $p_1$ is evidently birational and proper (because $Y$ is proper), so in light of $X$'s normality it suffices to show that $p_1$ is quasi-finite. 
        
        The rule $(x,y) \mapsto (x,\beta(y))$ defines a morphism $\Gamma_{\alpha} \to X \times Z$ which must factor through $\Gamma_{\beta \circ \alpha}$ because it does so on a non-empty open subset of $\Gamma_{\alpha}$.  The induced map $\Gamma_{\alpha} \to \Gamma_{\beta \circ \alpha}$ is quasi-finite because $\beta$ is quasi-finite.   But the  first projection $\Gamma_{\beta \circ \alpha} \to X$ is an isomorphism (because $\beta \circ \alpha$ is a morphism) and in particular quasi-finite, and hence $p_1 : \Gamma_{\alpha} \to X$, which is the composite $\Gamma_{\alpha} \to \Gamma_{\beta \circ \alpha} \to X$, is quasi-finite as promised. The lemma follows.
    \end{proof}
    
    \begin{remark}
        \label{remark:grassmannian}
        \begin{enumerate}
            \item We will use \Cref{lemma:extension} in a situation where $Y$ is a Grassmannian: every non-constant morphism out of a Grassmannian is necessarily finite (because the Grassmannian's Picard group is $\Z$), and so the critical hypothesis of \Cref{lemma:extension} will be satisfied.    
            \item \Cref{lemma:extension} is true in arbitrary characteristic. 
        \end{enumerate}
    \end{remark}

    \subsection{Resolution of indeterminacy and its cycle class}
    \label{subsection:resolution}

    The problem of determining the cycle class of graph closures of rational maps to projective spaces has, in a sense, a complete solution thanks to Aluffi in \cite[~Section 2]{aluffi2015multidegrees}.  Making this explicit in particular cases is difficult, and in this section we will give a self-contained treatment in two very special situations.  These situations will correspond to the $\gcd(m,n) = 1$ and $\gcd(m,n) = 2$ categories mentioned in the introduction. 

    In the following, if $\alpha$ is an element in a Chow ring, then $[\alpha]_{k}$ denotes the degree $k$ part of $\alpha$.  Furthermore, when we write $1/(1+\alpha)$, we mean the power series expansion of $1/(1+\alpha)$ in the variable $\alpha$.

    \begin{proposition}
        \label{proposition:resolution}
        Suppose $X$ is a smooth variety, and suppose $\L$ is a line bundle on $X$. Suppose $S = \{s_1, \dots, s_r\}$ and $T = \{t_1, \dots, t_{r+1}\}$ are sets of global sections of $\L$ such that the common vanishing schemes $Y_S$ and $Y_T$ of $S$ and $T$, respectively, have pure codimension $r$ in $X$.
        Then: 
        \begin{enumerate}
            \item The blow up $X_{S}$ (resp. $X_{T}$) of $X$ along $Y_S$ (resp. $Y_T$) is a closed subscheme of $X \times \P \langle S \rangle^{*}$ (resp. $X \times \P \langle T \rangle^{*})$.
            \item $X_{S}$ is the vanishing scheme of a section of the vector bundle $\L \boxtimes \mathscr{Q}^{r-1}$ on $X \times \P \langle S \rangle^{*}$, where $\mathscr{Q}^{r-1}$ is the rank $r-1$ universal quotient bundle over $\P \langle S \rangle^{*}$.
            \item The fundamental class $[X_{S}] \in \CH^{r-1}(X \times \P \langle S \rangle^{*})$ is given by 
            \begin{equation}
                \label{equation:classS}
                [X_{S}] = \left[\frac{(1+\lambda)^{r}}{(1 + \lambda - \zeta)}\right]_{r-1},
            \end{equation}
            where $\lambda$ (resp. $\zeta$) is the first Chern class of $\L$ (resp. hyperplane class) on  $X \times \P \langle S \rangle^{*}$ pulled back  from $X$ (resp. $\P \langle S \rangle^{*})$.
            \item Let $Y_1, Y_2, \dots, Y_{k}$ denote the irreducible components of $Y_T$, each with reduced induced scheme structure.  Then there exist positive integers $e_1, \dots, e_k$ such that 
            \begin{align}
                \label{equation:classT}
                \nonumber [X_{T}] = c_{r}(\L \boxtimes \mathscr{Q}^{r}) - &\sum_{i=1}^{k} e_i [Y_i \times \P \langle T \rangle^{*}]\\ 
                 & = \left[\frac{(1+\lambda)^{r+1}}{(1 + \lambda - \zeta)}\right]_{r} - \sum_{i=1}^{k} e_i [Y_{i} \times \P \langle T \rangle^{*}],
            \end{align}
            where $\mathscr{Q}^{r}$ is the rank $r$ universal quotient bundle over $\P \langle T \rangle^{*}$.
            \item If $W \subset \langle T \rangle$ is a general codimension $1$ subspace of sections, then the integer $e_i$ from part (4) is the multiplicity of the vanishing scheme $V(W)$ at a general point of $Y_i$. 
        \end{enumerate}
    \end{proposition}
    
    \begin{proof}
        \begin{enumerate}
        \item The sections $s_i$ together define a surjection \[\sigma: \L^{-1}\otimes \langle S \rangle \to \I_{Y_{S}}.\] This surjection then defines a surjective map of graded $\O_X$-algebras \[\Proj \Sym^{\bullet}(\L^{-1}\otimes \langle S \rangle) \to \Proj (\O_X \oplus \I_{Y_{S}} \oplus \I_{Y_{S}}^{2} \oplus \dots),\] which defines the desired closed embedding $X_{S} \hookrightarrow X \times \P \langle S \rangle^{*}$. 
        \item   The global sections $s_1 ,\dots, s_r$ define a map of sheaves \[\L^{-1} \to \O_{X} \otimes \langle S \rangle^{*}\] on $X$.   At the same time, over the projective space $\P \langle S \rangle^{*}$ we have the tautological sequence 
        \begin{equation}
            \label{equation:tautological}
            0 \to \O(-1) \to \O \otimes \langle S \rangle^{*} \to \mathscr{Q}^{r-1} \to 0.
        \end{equation}
        Pulling these sequences back to $X \times \P \langle S \rangle^{*}$, and suppressing unsavory pull-back notation, we obtain on $X \times \P \langle S \rangle^{*}$ a homomorphism \[s : \L^{-1} \to \mathscr{Q}^{r-1}.\] We may equivalently view $s$ as a section of the vector bundle $\L \boxtimes \mathscr{Q}^{r-1}$.  The description  of the embedding $X_S \hookrightarrow X \times \P \langle S \rangle^{*}$, in the proof of part (1) shows that $X_S$ is the vanishing scheme of $s$.
        
        Now we specialize only to the case of $S$, not $T$. By Krull's principal ideal theorem, every irreducible component of $V(s)$ has codimension at most $r-1$ in $X \times \P \langle S \rangle^{*}$.  We have already seen that $X_{S}$ is an irreducible component of $V(s)$.  The variety $(Y_{S})_{\red} \times \P \langle S \rangle^{*}$  has codimension $r$ in $X \times \P \langle S \rangle^{*}$, and so it is not an irreducible component of $V(s)$.  Therefore, $X_{S}$ is the unique irreducible component of $V(s)$, and assertion (2) follows.
    
        \item This follows immediately from part (2), and the fact that the cycle class of $V(s)$ is the top Chern class of the corresponding vector bundle $\L \boxtimes \mathscr{Q}^{r-1}$.  To compute this class, and to see why it is what has been indicated in the statement, simply tensor \eqref{equation:tautological} with $\L$ and use the Whitney sum formula.
        
        \item Each set $Y_{i} \times \P \langle T \rangle^{*}$ has codimension $r$ in $X \times \P \langle T \rangle^{*}$, and so the equality of sets \[V(s)_{\red} = X_{T} \cup \left((Y_{T})_{\red} \times \P \langle T \rangle^{*}\right) = X_{T} \cup (\cup_{i} Y_{i} \times \P \langle T \rangle^{*})\] implies an equality of cycles of the form \[[V(s)] = [X_{T}] + \sum_{i} e_i [Y_i \times \P \langle T \rangle^{*}].\]  Since $V(s)$ attains its expected codimension $r$, the cycle class of $V(s)$ is the top Chern class of the corresponding rank $r$ vector bundle $\L \boxtimes \mathscr{Q}^{r}$.  The assertion in part (4) follows.
        \item A general codimension $1$ subspace $W \subset \langle T \rangle$ is equivalently a general point $p \in \P \langle T \rangle^{*}$. Let $X_{p}$ denote the subvariety $X \times \{p\} \subset X \times \P \langle T \rangle^{*}$.  Let $y \in Y_{i}$ be a general point. Then the point $(y,p)$ is not contained in $X_{T}$, nor is it contained in the sets $Y_{j} \times \P \langle T \rangle^{*}$ for $j \neq i$.  
        
        Now consider the intersection $X_{p} \cap V(s)$, $s$ being the section of $\L \boxtimes \mathscr{Q}^{r}$ from the proof of part (2). This intersection has an equivalent description: upon identifying $X_p$ with $X$ in the obvious way, $X_{p} \cap V(s)$ is the vanishing scheme $V(W)$. The claim now follows by intersecting the equation \eqref{equation:classT} with the cycle $[X_{p}]$.
    \end{enumerate}
    \end{proof}

    \subsection{Basic properties of the first transvectant}

    \begin{definition}
        \label{definition:transvectant} Let $f$ and $g$ denote binary forms of degrees $m$ and $n$, respectively. The \textbf{first transvectant of $f$ and $g$}, denoted $$\{f,g\}$$ is the degree $m+n-2$ form $f_{x}g_{y} - f_{y}g_{x}$, where subscripts denote partial differentiation.
    \end{definition}

    The first transvectant is fundamental in the invariant theory of binary forms.  Most importantly, it is an $\SL(2)$-equivariant bilinear map $S_{m} \times S_{n} \to S_{m+n-2}$.  Therefore, it induces a $\GL(2)$-equivariant rational map $\P S_m \times \P S_{n} \dashrightarrow \P S_{m+n-2}$, and in the next proposition we will identify the locus of indeterminacy of this rational map.

   \begin{proposition}
       \label{lemma:indeterminacy}
       Suppose $f$ and $g$ are two nonzero homogeneous forms of degrees $m$ and $n$ respectively. If $\{f,g\} = 0$, then there is a form $h$ of degree $\gcd(m, n)$ such that $f$ and $g$ are both powers of $h$, up to scaling.  
   \end{proposition}
   \begin{proof}
       We begin with the case that $m = n$. In this case, consider the map $\rho: \P^1 \to \P^{1}$ given by $[x:y] \mapsto [f:g]$. If $h$ is the greatest common divisor of the forms $f$ and $g$, and if $\overline{f}$ and $\overline{g}$ denote $f/h$ and $g/h$, then $\rho$ is also given by the map $[x:y] \mapsto [\overline{f}:\overline{g}]$. The condition $\{f,g\} = 0$ is equivalent to the vanishing of the derivative of $\rho$ at a general point of $\P^{1}$. Therefore, $\rho$ is constant, so $f$ and $g$ are scalar multiples of each other, and the proposition is proved in this case.
   
       Now let us assume $m$ and $n$ are not necessarily equal.  Then there are relatively prime integers $k$ and $l$ such that $f^{k}$ and $g^{l}$ have the same degree. But if $\{f,g\} = 0$, then it follows that $\{f^{k},g^{l}\} = 0$ as well, because $\{f^{k},g^{l}\} = klf^{k-1}g^{l-1}\{f,g\}$. Therefore, by the case settled in the first paragraph,  $f^{k}$ and $g^{l}$ must be scalar multiples of each other. In particular, $f$ and $g$ vanish at the same set of points, which we denote by $S$.
   
       Suppose, then, that $\sum_{p \in S} \ord_{p}(f) p$ and $\sum_{p \in S} \ord_{p}(g) p$ are the divisors of vanishing of $f$ and $g$.  We just observed that $\sum_{p \in S} (k \ord_{p}(f)) p = \sum_{p \in S} (l \ord_{p}(g)) p$. By co-primality of $k$ and $l$, we see that $l \mid \ord_{p}(f)$ and $k \mid \ord_{p}(g)$ for all $p \in S$. Let $h$ be a form such that the vanishing divisor of $h$ is $\sum_{p \in S} (\ord_{p}(f)/l) p = \sum_{p \in S} (\ord_{p}(g)/k) p$. Then $f = h^{l}$ and $g = h^{k}$, up to scale.  The fact that $h$ has degree $\gcd(\deg f, \deg g)$ is also immediate, completing the proof.
   \end{proof}

   \begin{remark}
    \label{remark:indeterminacycharacteristic}
    In the proof of \Cref{lemma:indeterminacy}, we deduced the vanishing of the derivative of a morphism from its constancy. Therefore, \Cref{lemma:indeterminacy} critically uses the characteristic zero assumption on $\k$.
\end{remark}

\begin{corollary}
    \label{corollary:wronskiregular}
    The map $\upsilon: \Gr(2,S_{d}) \dashrightarrow \P S_{2d-2}$ which sends a pencil $p \wedge q$ to the point $\{p,q\}$ is regular. 
\end{corollary}

\begin{proof}
    This follows from the $m = n = d$ case of \Cref{lemma:indeterminacy}. 
\end{proof}

\section{Introducing the variety of first transvectants} 
\label{section:main}

Let us now fix positive integers $m, n, a, b, d$ such that $a m = b n = d$.  Consider the following commutative diagram:

\begin{center}
    \begin{equation}\label{diagram:clebsch}
\begin{tikzcd}[row sep=large, column sep=large]
    \P S_{m} \times \P S_{n} \arrow[r, "\varphi_{a,b}", dashed] \arrow[d,"\epsilon", dashed] & \Gr(2,S_{d}) \arrow[d, "{\upsilon}"]\\
    \P S_{m(a-1) + n(b-1)} \times \P S_{m+n-2} \arrow[r, "\mult"] & \P S_{2d-2}
\end{tikzcd}
    \end{equation}
\end{center}

Here
\begin{enumerate}
    \item $\varphi_{a,b}$ is the rational map $(f,g) \mapsto f^a \wedge g^b$, 
    \item $\epsilon$ is the rational map $(f,g) \mapsto (f^{a-1}g^{b-1},\{f,g\})$,
    \item $\mult$ is the multiplication of forms, and
    \item $\upsilon$ is the map sending a pencil $p \wedge q$ to $\{p,q\}$, which is regular by \Cref{corollary:wronskiregular}.
\end{enumerate}

Clearly, if we can explicitly describe and work with a resolution of indeterminacy of $\varphi_{a,b}$, it would get us most of the way towards solving the enumerative problem of counting sums of $a$'th powers and $b$'th powers.  

Next observe that to resolve the indeterminacy of $\epsilon$ we need only resolve the indeterminacy of the composite of $\epsilon$ with projection to the factor $\P S_{m+n-2}$, because the composite with the projection to the first factor is regular.  This motivates the following definition:

\begin{definition}
    \label{definition:transvectantvariety}
    The \textbf{variety of first transvectants} is the variety 
    \begin{equation}\label{equation:beta}
    \beta: \transvectant_{m,n} \to \P S_{m} \times \P S_{n} \times \P S_{m+n-2}
    \end{equation} obtained by normalizing the closure of the graph of the first transvectant map $\P S_{m} \times \P S_{n} \dashrightarrow \P S_{m+n-2}$ given by $(f,g) \mapsto \{f,g\}$. 
\end{definition}

A general point of $\transvectant_{m,n}$ represents a triple $(f,g,t)$ of binary forms where $t$ is the first transvectant of $f$ and $g$, all three considered up to scaling.  It also contains limits of such triples as $\{f,g\}$ tends to $0$.

\begin{remark}
    \label{remark:schemestructure}
    The base scheme $Y \subset \P S_{m} \times \P S_{n}$ of the first transvectant map $\{\,,\,\}: \P S_{m} \times \P S_{n} \dashrightarrow \P S_{m+n-2}$ is the common vanishing scheme of the $m+n-1$ coefficients of the transvectant $\{f,g\}$, which are quadratic polynomials of bi-degree $(1,1)$ in the coefficients of $f$ and $g$. 
\end{remark}

\begin{proposition}
    \label{proposition:clebsch-resolved}
    \begin{enumerate}
        \item The rational map $\varphi_{a,b}$ extends to a morphism on $\transvectant_{m,n}$.
        \item $\transvectant_{m,n}$ inherits an action of $\GL(2)$ making $\beta$ an equivariant morphism.
    \end{enumerate}
\end{proposition}

\begin{proof} 
    \begin{enumerate} 
        \item First note that $\epsilon$ extends to a morphism on $\transvectant_{m,n}$. Next we apply \Cref{lemma:extension} with $X = \transvectant_{m,n}$, $Y = \Gr(2,S_{d})$, $Z = \P S_{2d-2}$, $\alpha = \varphi_{a,b}$, and $\beta = \upsilon$, along with the commutativity of \eqref{diagram:clebsch} to reach the desired conclusion. 
       
        \item This is true because the rational map $$\{,\}:\P S_{m} \times \P S_{n} \dashrightarrow \P S_{m+n-2}$$ is originally $\GL(2)$-equivariant, and so the closure of its graph is $\GL(2)$-invariant.  Finally, an action of a group on a variety automatically lifts to an action of the group on its normalization, and so the claim follows.
    \end{enumerate}
\end{proof}

The picture we obtain from \Cref{proposition:clebsch-resolved} is now:

\begin{center}
    \begin{equation}\label{equation:clebsch-resolved}
    \begin{tikzcd}[row sep=large, column sep=large]
        \transvectant_{m,n} \arrow[from=1-1, to=2-2, "\beta_{12}"] \arrow[from=1-1, to=2-3, "\widetilde{\varphi_{a,b}}", bend left = 10] \arrow[from=1-1, to=3-2, "\widetilde{\epsilon}", bend right = 20] &  & \\
    & \P S_{m} \times \P S_{n} \arrow[r, "\varphi_{a,b}", dashed] \arrow[d, "\epsilon", dashed] & \Gr(2,S_{d}) \arrow[d, "\upsilon"]\\
        & \P S_{m(a-1) + n(b-1)} \times \P S_{m+n-2} \arrow[r, "\mult"] & \P S_{2d-2}
    \end{tikzcd}
    \end{equation}
\end{center}

Here $\beta_{12}$ is by definition the composite $\transvectant_{m,n} \xrightarrow{\beta} \P S_{m} \times \P S_{n} \times \P S_{m+n-2} \xrightarrow{p_{12}} \P S_{m} \times \P S_{n}$, and $\widetilde{\epsilon}$ is the extension of $\epsilon$ to $\transvectant_{m,n}$.  Observe that $\beta_{12}$ is a birational morphism.  Diagram \eqref{equation:clebsch-resolved} shows how $\transvectant_{m,n}$ is crucial to the problem of counting sums of two powers, and so we will study some of its enumerative geometry next.

\section{Enumerative geometry of $\transvectant_{m,n}$}
\label{section:enumerative}

To understand $\transvectant_{m,n}$ in any depth, we must introduce coordinates.  To that end, we will denote by 
\begin{align*}
    a = a_{0}x^{m} + {e \choose 1} a_{1}x^{m-1}y + {e \choose 2} a_{2}x^{m-2}y^{2} + \dots + a_{e}y^{e}
\end{align*} a general element of $\P S_{e}$.  If 

\begin{align*}
f &= f_{0}x^{m} + {m \choose 1} f_{1}x^{m-1}y + {m \choose 2} f_{2}x^{m-2}y^{2} + \dots + f_{m}y^{m}\\
g &= g_{0}y^{n} + {n \choose 1} g_{1}x^{n-1}y + {n \choose 2} g_{2}x^{n-2}y^{2} + \dots + g_{n}x^{n}
\end{align*}
are forms of degree $m$ and $n$, respectively, then the transvectant $t = \{f,g\}$ is a form of degree $m+n-2$ whose coefficients are described in the next lemma.

\begin{lemma}
    \label{lemma:transvectantcoefficients}
    Let $t = t_{0}x^{m+n-2} +  t_{1}x^{m+n-3}y + t_{2}x^{m+n-4}y^{2} + \dots + t_{m+n-2}y^{m+n-2}$ be the transvectant of $f$ and $g$. 
    \begin{enumerate}   
        \item The coefficient $t_k$ is a homogeneous sum of monomials of the form $f_i g_j$ for $i+j = k+1$. Every monomial $f_i g_j$ satisfying $i+j = k+1$ appears with non-zero coefficient in $t_k$.
        \item If $f_i, g_j, f_j, g_i$ all exist, then the coefficients of $f_i g_j$ and $f_j g_i$ in $t_{i+j-1}$ sum to zero.
    \end{enumerate}
\end{lemma}

\begin{proof}
    This follows from direct calculation of the first transvectant.
\end{proof}

\subsection{A bit of equivariant geometry}
\label{subsection:equivariant}

Throughout this section, we fix positive integers $m,n,a,b,d$ satisfying $am = bn = d$.

Now we will look at the $T$-equivariant geometry of $\transvectant_{m,n}$ and the morphism $\widetilde{\varphi_{a,b}}$ in \eqref{equation:clebsch-resolved}, where $$T \subset \GL(2)$$ is the $1$-dimensional torus acting on $x$ with weight $1$ and acting on $y$ with weight $-1$. For any integer $u$ we let $\k(u)$ denote the $1$-dimensional representation of $T$ having weight $u$. If $e \geq 1$ is a positive integer, then $T$ acts on $S_e$ by acting on the monomial $x^iy^{e-i}$ with weight $2i-e$.

Observe that $\P S_{e}$ has only finitely many $T$-fixed points, namely the degree $e$ monomials $x^iy^{e-i}$ for $i = 0, \dots, e$. Therefore, since $\transvectant_{m,n}$ is the normalization of a closed $T$-invariant subscheme of $\P S_{m} \times \P S_{n} \times \P S_{m+n-2}$, it also has only finitely many $T$-fixed points. 

Let us denote by $$\S(d) \to \transvectant_{m,n}$$ the pullback of the tautological rank $2$ bundle under the map $\widetilde{\varphi_{a,b}}$ from \eqref{equation:clebsch-resolved}.  Our immediate objective is to determine the two $T$-weights of the fiber of $\S(d)$ at an arbitrary $T$-fixed point $t \in \transvectant_{m,n}$.  It follows that \[\beta(t) = (x^{i}y^{m-i}, x^{j}y^{n-j}, x^{k}y^{m+n-2-k}) \in \P S_{m} \times \P S_{n} \times \P S_{m+n-2}\] for some integers $i,j,k$. 

First we make a definition for convenience, naturally separating the fixed points into two types. 

\begin{definition}
    \label{definition:exceptionalfixedpoint} A $T$-fixed point $t \in \transvectant_{m,n}$ is \textbf{exceptional} if $t$ is an element of the non-isomorphism locus of $\beta_{12}: \transvectant_{m,n} \to \P S_{m} \times \P S_{n}$.  A $T$-fixed point is \textbf{unexceptional} otherwise.
\end{definition}

The restriction of $\S(d)$ to the $T$-fixed point $t$ is a $T$-representation, and hence is a direct sum of two characters.  The next proposition determines these characters. 

\begin{proposition}
\label{proposition:weightsofS}
Let $t \in \transvectant_{m,n}$ be a $T$-fixed point with $\beta(t) = (x^{i}y^{m-i}, x^{j}y^{n-j}, x^{k}y^{m+n-2-k})$ the corresponding point in $\P S_{m} \times \P S_{n} \times \P S_{m+n-2}$.  Then 
\[ \S(d)|_{t} = \k(2ai - d) \oplus \k((2bj - d) + 2k-2i-2j+2)\] as $T$-representations.
\end{proposition}

\begin{proof}
    There are two cases to consider: (1) $t$ is unexceptional, or (2) $t$ is exceptional.

    In case (1), the claim is obvious from the description of $\varphi_{a,b}$, considering that $k$ must equal $i+j-1$ when $t$ is unexceptional.  
    
    In case (2), we first make an observation: The point $\widetilde{\varphi_{a,b}}(t) \in \Gr(2,S_{d})$ is contained in the subset of pencils having $(x^{i}y^{m-i})^{a}$ (which is equal to $(x^{j}y^{n-j})^{b}$ in this case) as a member. This is because the map $\widetilde{\varphi_{a,b}}$ factors through the map $\widetilde{\varphi_{1,1}}: \transvectant_{d,d} \to \Gr(2,S_{d})$.
    
    As $\widetilde{\varphi_{a,b}}(t)$ is $T$-fixed, it follows that \[\widetilde{\varphi_{a,b}}(t) = (x^{i}y^{m-i})^{a} \wedge x^{u}y^{v}\] for some $u,v$. To extract more information, observe that the point $\upsilon(\widetilde{\varphi_{a,b}}(t))$ can be read directly from the point $t$ by going the other way around diagram \eqref{equation:clebsch-resolved}. The result is \[\upsilon(\widetilde{\varphi_{a,b}}(t)) = (x^{i}y^{m-i})^{a-1}(x^{j}y^{n-j})^{b-1}(x^{k}y^{m+n-2-k}) \in \P S_{2d-2}.\]  On the other hand, by directly computing the transvectant $\{(x^{i}y^{m-i})^{a} , x^{u}y^{v} \}$ we find that 
    \begin{align*}
        u &= -(i-1) + (b-1)j + k\\
        v &= -(m-i-1) + (b-1)(n-j) + (m+n-2-k)
    \end{align*}
    Therefore, the weight of the monomial $x^{u}y^{v}$ is $(2bj-d)+2k-2i -2j+2$ and the exceptional case of the proposition follows.
\end{proof}

Though \Cref{proposition:weightsofS} is quite elementary, the next few corollaries show how it points the way to the general solution to the problem of counting sums of powers.  

\begin{corollary}
\label{corollary:samerestriction}
Write $\O(a_1,a_2,a_3)$ for the line bundle obtained by pulling back $\O(a_{i})$ from the $i$-th factor of $\P S_{m} \times \P S_{n} \times \P S_{m+n-2}$. The vector bundle $\S(d)$ on $\transvectant_{m,n}$ restricts to each $T$-fixed point of $\transvectant_{m,n}$ in the same way as \[\beta^{*}\left(\O(-a,0,0) \oplus \O(1,1-b,-1)\right)\] as $T$-representations. 
\end{corollary}

\begin{proof}
    This is immediate from \Cref{proposition:weightsofS}.
\end{proof}


\begin{corollary}
    \label{corollary:intersectionnumbers}
    For $i = 1,2,3$ let $\zeta_{i}$ denote the divisor class on $\P S_{m} \times \P S_{n} \times \P S_{m+n-2}$ obtained by pulling back the hyperplane class from the $i$-th factor.  Suppose $p(\sigma_1,\sigma_2) \in \Z[\sigma_1,\sigma_2]$ is a homogeneous degree $m+n$ polynomial in the two Chern classes $\sigma_1,\sigma_2$ of $\mathscr{S}(d)$ where $\sigma_j$ has degree $j$.   Finally, set 
    \begin{align*}
        \alpha_{1} &= (-a+1)\zeta_1 + (-b+1)\zeta_2 - \zeta_3\\
        \alpha_{2} &= -a\zeta_1 \cdot \left(\zeta_1 + (-b+1)\zeta_2 - \zeta_3\right).
    \end{align*}
    
    Then we get the following equality of integers: \[\int_{\transvectant_{m,n}} p(\sigma_1,\sigma_2) = \int_{\P S_{m} \times \P S_{n} \times \P S_{m+n-2}} p(\alpha_1,\alpha_2) \cap \beta_{*}[\transvectant_{m,n}].\]
\end{corollary}

\begin{proof}
    This is a consequence of \Cref{corollary:samerestriction} and localization with respect to the torus action (see for instance \cite{edidin1998localization}).  Specifically, let $\rho: \transvectant_{m,n}^{'} \to \transvectant_{m,n}$ be a $T$-equivariant resolution of singularities. Then the bundles $\rho^{*}\S(d)$ and $\rho^{*}\beta^{*}\O(-a,0,0) \oplus \rho^{*}\beta^{*}\O(1,1-b,-1)$ have the same $T$-equivariant restrictions to each $T$-fixed subvariety. Therefore, torus localization with respect to the $T$-action on $\transvectant_{m,n}'$ implies $$\int_{\transvectant_{m,n}^{'}} \rho^{*} p(\sigma_1,\sigma_2) = \int_{\transvectant_{m,n}^{'}} \rho^{*} p(\alpha_1,\alpha_2),$$ and then the corollary follows from pushing this equality forward to $\P S_{m} \times \P S_{n} \times \P S_{m+n-2}$.
\end{proof}

\begin{corollary}
    \label{corollary:solutioninprinciple} 

    Let $\alpha_{1}$ and $\alpha_2$ denote the classes found in \Cref{corollary:intersectionnumbers}.  Write $\gamma \in \CH^{m+n} \P S_{m} \times \P S_{n} \times \P S_{m+n-2}$ for the degree $m+n$ term of the formal sum $\sum_{i=0}^{\infty} (-\alpha_{1}-\alpha_{2})^{i}$.  

    Then the locus of forms in $\P S_{d}$ which are the sum of an $a$'th power and a $b$'th power is the intersection product \[\int_{\P S_{m} \times \P S_{n} \times \P S_{m+n-2}} \gamma \cap \beta_{*}[\transvectant_{m,n}].\]
\end{corollary}

\begin{proof}
    This is a direct consequence of \Cref{corollary:intersectionnumbers} and the following simple fact which we leave to the reader: {\sl If $\iota: B \to \Gr(2,N+1)$ is a morphism to a Grassmannian from a projective variety $B$, then $$\int_{B} \left[ \frac{1}{c(\iota^{*} S)} \right]_{\dim B}$$ is the degree of the variety in $\P^{N}$ swept out by the lines parametrized by $\iota$.} (Here $S$ is the tautological rank $2$ bundle on $\Gr(2,N)$, and $c(S)$ is its Chern polynomial.)
\end{proof}

\subsection{Computing $\beta_{*}[\transvectant_{m,n}]$ in two special cases}
\label{subsection:gcds}

\begin{theorem}
    \label{theorem:czvclass} Let $m, n$ be two positive integers satisfying $m \leq n$ and consider the variety of first transvectants $\beta: \transvectant_{m,n} \to \P S_{m} \times \P S_{n} \times \P S_{m+n-2}$. For $i \in \{1,2,3\}$ let $$\zeta_{i} \in \CH^{1}(\P S_{m} \times \P S_{n} \times \P S_{m+n-2})$$ denote the pullback of the hyperplane class from the $i$-th factor. 
    \begin{enumerate}
        \item If $\gcd(m,n) = 1$, then 
        \begin{equation}
            \label{equation:czvclass1}
        \beta_{*}[\transvectant_{m,n}] = \left[\frac{(1+\zeta_{1}+\zeta_{2})^{m+n-1}}{1+\zeta_{1}+\zeta_{2}-\zeta_{3}} \right]_{m+n-2}.
        \end{equation} 
        \item  If $\gcd(m,n) = 2$, then 
        \begin{align}    
            \label{equation:czvclass2}
        \nonumber \beta_{*}[\transvectant_{m,n}] & =  \left[\frac{(1+\zeta_{1}+\zeta_{2})^{m+n-1}}{1+\zeta_{1}+\zeta_{2}-\zeta_{3}} \right]_{m+n-2} \\
        & - 2^{m-2}\left(\left(\frac{m}{2}\right)^{2}\zeta_{1}^{m-2}\zeta_{2}^{n} + \frac{n}{2}\frac{m}{2}\zeta_{1}^{n-1}\zeta_{2}^{m-1} + \left(\frac{n}{2}\right)^{2}\zeta_{1}^{m}\zeta_{2}^{n-2}\right).
        \end{align}    
    \end{enumerate}
    \end{theorem}
    
    \begin{proof} Recall that, by its very construction, $\beta(\transvectant_{m,n})$ is the closure of the graph of the first transvectant map.  We will refer to \Cref{proposition:resolution} in the special case where: 
\begin{enumerate}
    \item $X = \P S_{m} \times \P S_{n}$
    \item $Y = \beta(\transvectant_{m,n}) \subset X \times \P S_{m+n-2}$
    \item $\L = p_1^{*}\O(1) \otimes p_2^{*}\O(1)$ on $X$ -- the first transvectant map is given by bi-degree $(1,1)$ polynomials.
\end{enumerate}

Observe: 
\begin{enumerate}
    \item[(A)] When $\gcd(m,n) = 1$, the scheme $Y \subset \P S_{m} \times \P S_{n}$ defined by $\{f,g\} = 0$ is the common vanishing scheme of $m+n-1$ sections of $\L$ {\sl and has codimension $m+n-1$}, by \Cref{lemma:indeterminacy}. This falls within the scope of the ``S'' case of \Cref{proposition:resolution}.
    \item[(B)] When $\gcd(m,n) = 2$, the scheme $Y \subset \P S_{m} \times \P S_{n}$ defined by $\{f,g\} = 0$ is the common vanishing scheme of $m+n-1$ sections of $\L$ {\sl and has codimension $m+n-2$}, by \Cref{lemma:indeterminacy}. This falls within the scope of the ``T'' case of \Cref{proposition:resolution}. 
\end{enumerate}

We now take both statements in the theorem in turn.

    \begin{enumerate}
        \item This is immediate from part (3) of \Cref{proposition:resolution} in light of observation (A) above.
        \item Part (4) of \Cref{proposition:resolution} applies thanks to observation (B) above.  We conclude that there exists a positive integer $e$ such that 
        \begin{equation}
            \beta_{*}[\transvectant_{m,n}] = \sum_{i=0}^{m+n-2} (\zeta_{1}+\zeta_{2})^{i} \cdot \zeta_{3}^{m+n-2-i} - e[Z],
        \end{equation}
        where $Z$ is the set of triples $(q^{m/2},q^{n/2},p)$ as $q \in \P S_{2}$ and $p \in \P S_{m+n-2}$ vary freely.  The cycle class of $Z$ is easily determined to be \[\left(\frac{m}{2}\right)^{2}\zeta_{1}^{m-2}\zeta_{2}^{n-2} + \frac{m}{2}\frac{n}{2}\zeta_{1}^{m-1}\zeta_{2}^{n-1} + \left(\frac{n}{2}\right)^{2}\zeta_{1}^{m}\zeta_{2}^{n-2}.\]  Therefore, it remains to show that $e = 2^{m-2}$, and this is what we now turn to.
    
        Part (5) of \Cref{proposition:resolution} tells us how to compute $e$.  Let 
        \begin{align*}
        f = f_{0}x^{m} + {m \choose 1} f_{1}x^{m-1}y + {m \choose 2} f_{2}x^{m-2}y^{2} + \dots + {m \choose m-1} f_{m-1}xy^{m-1} + f_{m}y^{m},\\
        g = g_{0}x^{n} + {n \choose 1} g_{1}x^{n-1}y + {n \choose 2} g_{2}x^{n-2}y^{2} + \dots + {n \choose n-1} g_{n-1}xy^{n-1} + g_{n}y^{n}
        \end{align*}
        denote general binary forms of degrees $m$ and $n$.   Then let 
        \begin{align*}
            t = t_0 x^{m+n-2} + t_{1}x^{m+n-3}y + \dots + t_{m+n-2}y^{m+n-2}
        \end{align*}
        denote the transvectant form $\{f,g\}$. The following claims are immediate from the definition of the transvectant: 
        \begin{enumerate}
            \item Each $t_{k}$ is a homogeneous sum of monomials of the form $f_{i}g_{j}$.
            \item The monomials $f_{i}g_{j}$ occurring in $t_{k}$ are precisely those which satisfy $i+j = k+1$.
            \item If $f_{i}, g_{j}, f_{j}, g_{i}$ all exist, then the coefficients of $f_{i}g_{j}$ and $f_{j}g_{i}$ in $t_{i+j-1}$ differ only in sign.
        \end{enumerate}
    
        Let $Y \subset \P S_{m} \times \P S_{n}$ denote the base-scheme of the transvectant map, and denote by $Y_{\red}$ its underlying reduced scheme. Then, thanks to \Cref{lemma:indeterminacy}, $Y_{\red}$ is isomorphic to $\P S_{2}$ because of our assumption that $\gcd(m,n) = 2$.  Here, the point in $Y_{\red}$ corresponding to a quadratic form $q \in \P S_{2}$ is the point $(q^{m/2}, q^{n/2})$.  The action of $\GL(2)$ on $\P S_{2}$ is transitive on the open dense set of quadratic forms with distinct roots, and so, by $\GL(2)$-equivariance, it suffices to compute the integer $e$ (following the recipe inherent in part (5) of \Cref{proposition:resolution}) at the specific point $u := (x^{m/2}y^{m/2}, x^{n/2}y^{n/2}) \in Y$ corresponding to $xy \in \P S_{2}$. 
    
        Let $U \subset \P S_{m} \times \P S_{n}$ be the affine chart around $u$ obtained by dehomogenizing by setting $f_{m/2} = g_{n/2} = 1$. Then the remaining coefficient variables $f_i, g_j$ are interpreted as affine coordinates on $U$ around the point $u$. Part (5) of \Cref{proposition:resolution} says we must compute the multiplicity of $V(W) \cap U$ at $u$, where $V(W)$ is the scheme cut out by a general codimension $1$ subspace $W \subset \k \langle t_0, \dots, t_{m+n-2} \rangle$. 
    
        For this, we need to access the tangent cone $TC_{u}(V(W))$ and then need to compute the degree of its projectivization, the latter being a curve in the projectivized tangent space $\P T_{u} U$ (it is a curve because $V(W)$ is $2$-dimensional at $u$).  And so, we must analyze the lowest order homogeneous terms of the (now dehomogenized) polynomials $t_0, \dots, t_{m+n-2}$. In what follows, if $p$ is a polynomial, then $\init(p)$ will denote its initial part -- the homogeneous polynomial of lowest degree in $p$.
    
        Our first claim is that the middle polynomial $t_{m/2 +n/2 - 1}$ is homogeneous, and obviously quadratic. (We have dehomogenized by setting $f_{m/2} = g_{n/2} = 1$, so this statement has content.) This is because, in the expression $\{f,g\}$ the contribution to $t_{m/2 +n/2 - 1}$ coming from $x^{m/2}y^{m/2}$ and $x^{n/2}y^{n/2}$ have opposite signs and are equal.  
    
        Our next claim is that the first and last $m/2-1$ polynomials $t_0, \dots, t_{m/2-2}$ and $t_{m+n-2}, \dots, t_{m+n-m/2}$ are also homogeneous (and of course quadratic).  This is simply for index reasons, recalling the fact that the polynomial $t_k$ is a sum of monomials $f_{i}g_{j}$ with $i+j = k+1$. 
    
        Our third claim is that the remaining polynomials $t_{m/2-1}, \dots, t_{m/2+n-1}$, excluding the middle one $t_{m/2+n/2-1}$, begin with a nonzero linear term.  The linear term in $t_{k}$, for $k$ in this range, has nonzero contribution only from $g_{k+1-m/2}$ and, if available, $f_{k+1-n/2}$.  In particular, the $n+1$ linear forms $$\init(t_{m/2-1}), \dots, \init(t_{m/2+n-1})$$ (excluding $t_{m/2+n/2-1}$) are linearly independent. 
    
        Let $Q \subset \k \langle t_0, \dots, t_{m+n-2} \rangle$ be the subspace spanned by the homogeneous quadratic elements among the $t_{i}$ -- by what we have just seen, $Q$ has dimension $2(m/2-1)+1 = m-1$.  Therefore, $W \cap Q$ has dimension $m-2$ (because $W$ is general). Let $w_{1}, \dots, w_{m-2}$ be a basis of $W \cap Q$, and then let $w_{m-1}, \dots, w_{m+n-2}$ be a completion to a basis of $W$.  In particular, the elements $\init(w_{m-1}), \dots, \init(w_{m+n-2})$ are linearly independent linear forms (by generality of $W$).
    
        The common vanishing scheme of $\{w_{1}, \dots, w_{m-2},\init(w_{m-1}), \dots, \init(w_{m+n-2})\}$ is a curve in the projective space $\P T_{u}U$ (because $V(W)$ is a surface near $u$), and hence is the complete intersection of the listed forms. Therefore, this curve has degree $2^{m-2}$ by B\'ezout's theorem, establishing $e = 2^{m-2}$. The theorem is proved.
    \end{enumerate}
    \end{proof}

\section{Questions and Code}
\label{section:conclusion}

By combining \Cref{theorem:czvclass} with \Cref{corollary:solutioninprinciple} in the case $(m,n,a,b) = (2,3,3,2)$ and $(4,6,3,2)$, we recover Clebsch's $40$ and Zariski-Vakil's $3762$.  Applying the same formulas with $(m,n,a,b) = (3,5,5,3)$ and $(4,10,5,2)$ we get $29822$ and $626327$ as promised in the introduction.  We have not attempted to give a closed form for counting sums of two powers in the two $\gcd(m,n)$ cases we've considered, but we do provide at the end of this paper a Python script which computes the degree of the $f^{a} + g^{b}$ locus in these two cases.   

\begin{remark}
    \label{remark:entirecycle} 
    The patient reader will observe that \Cref{theorem:czvclass} combined with \Cref{corollary:solutioninprinciple} actually provide access to the entire cycle $\widetilde{\varphi}_{a,b *}[\transvectant_{m,n}] \in \CH_{m+n}(\Gr(2,S_{d}))$.  The count of sums of two powers is only one coefficient of this cycle when it is expanded with respect to the basis consisting of fundamental classes of Schubert varieties.  For simplicity of exposition, we chose not to include calculations of other coefficients in this note.
\end{remark}

\subsection{Questions}
\label{subsection:questions}

There are many more things to think about, beyond the obvious matter of solving the general problem, but here are a few of our favorites:

\begin{question}
    \label{question:general} What about counting sums of three or more powers? Are there similar mission-critical varieties like $\transvectant_{m,n}$ for these problems?
\end{question}

\begin{question}
    \label{question:higherforms} Sticking to two powers, what happens if we increase the number of variables?  For example, what is the degree of the locus of $f^{3} + g^{2}$ sextics in the $\P^{27}$ of sextic curves in the plane? 
\end{question}

\begin{remark}
    \label{remark:ZariskiPairs} The sextic curves mentioned in \Cref{question:higherforms} famously form one of the two irreducible components of the variety parametrizing sextics with six cusps.  These sextics are branch curves of linear projections of cubic surfaces.  The other component is more mysterious, and the author does not know any good presentation for their equations.  This is the starting point of the beautiful topic of Zariski pairs.
\end{remark}

\begin{remark}
    \label{remark:notdaunting} Although \Cref{question:higherforms} may seem daunting, we note that the base scheme of the rational map $\P^{5} \times \P^{9} \dashrightarrow \G(1,27)$ sending a general pair $(Q,C)$ consisting of a ternary quadric and cubic to the pencil $Q^{3} \wedge C^{2}$ is ill-defined precisely on the locus of pairs $(L^{2},L^{3})$ where $L$ is a linear form.  This locus consists of a single $\GL(3)$-orbit, and so there aren't any embedded primes in the indeterminacy {\sl scheme} of this rational map.  This provides considerable confidence in attempting to find a workable resolution.
\end{remark}

\begin{question}
    \label{question:fundamentalclassofacovariant}  The first transvectant $(f,d) \mapsto \{f,g\}$ is one of the most important examples of an $\SL(2)$-covariant of binary forms.  But there are plenty more covariants, including higher transvectants.  Each of these has a fundamental class, by which we mean the fundamental class of the graph closure as in \Cref{theorem:czvclass}.  What are these classes? 
\end{question}

\subsection{Code}
\label{subsection:code}

The following Python code, generated by Gemini, provides the counts of sums of two powers in the two $\gcd(m,n)$ cases we've considered. 

\begin{lstlisting}[basicstyle=\ttfamily\tiny]  
import math
from sympy import symbols, expand, Poly, factorial

def get_total_degree_part(poly, variables, degree):
    """
    Extracts the part of a polynomial that is homogeneous of a specific total degree.
    """
    if poly == 0:
        return 0
    poly_dict = Poly(poly, *variables).as_dict()
    new_poly = 0
    for powers, coeff in poly_dict.items():
        if sum(powers) == degree:
            term = coeff
            for var, p in zip(variables, powers):
                term *= var**p
            new_poly += term
    return new_poly

def compute_degree(m, n, a, b):
    """
    Computes the degree of the locus of sums of a-th powers and b-th powers.

    Args:
        m (int): Degree of the base form for the first power.
        n (int): Degree of the base form for the second power.
        a (int): The first power.
        b (int): The second power.

    Returns:
        int: The computed degree.
    """
    # Verify conditions from the paper 
    if a * m != b * n:
        raise ValueError("Condition am = bn is not satisfied.")

    common_divisor = math.gcd(m, n)
    if common_divisor not in [1, 2]:
        raise ValueError("Condition gcd(m, n) must be 1 or 2.")

    # Enforce m <= n as assumed in the paper's Theorem 4.7
    if m > n:
        m, n = n, m
        a, b = b, a

    # Define symbolic variables for the hyperplane classes zeta_1, zeta_2, zeta_3
    z1, z2, z3 = symbols('z1 z2 z3')
    variables = (z1, z2, z3)

    # 1. Compute the class gamma based on Corollary 4.5 and 4.6 

    # Define alpha_1 and alpha_2 
    alpha1 = (1 - a) * z1 + (1 - b) * z2 - z3
    alpha2 = -a * z1**2 - a * (1 - b) * z1 * z2 + a * z1 * z3

    # Compute gamma 
    gamma_poly = 0
    for j in range((m + n) // 2 + 1):
        i = m + n - 2 * j
        if i >= 0:
            coeff = ((-1)**(i + j) * factorial(i + j) // (factorial(i) * factorial(j)))
            term = coeff * (alpha1**i) * (alpha2**j)
            gamma_poly += term

    gamma_poly = expand(gamma_poly)

    # 2. Compute the class beta_*[Tvm,n] based on Theorem 4.7

    # First part of the formula, valid for gcd=1 and as a base for gcd=2 
    beta_star_poly = 0
    for i in range(m + n - 1): # Sum up to m+n-2
        term = expand((z1 + z2)**i) * (z3**(m + n - 2 - i))
        beta_star_poly += term

    # Add correction term for the gcd=2 case [cite: 254, 255]
    if common_divisor == 2:
        # Check if m is even, which it must be if gcd(m,n)=2 and m<=n
        if m % 2 != 0:
             # This case should not be reachable if gcd(m,n)=2 and m is the smaller
             raise ValueError("m must be even for gcd(m,n)=2 case.")

        # Correction term from Theorem 4.7 (8)
        # Using // for integer division
        corr_coeff = 2**(m - 2)
        term1 = (m // 2)**2 * z1**(m - 2) * z2**n
        term2 = (n // 2) * (m // 2) * z1**(m - 1) * z2**(n - 1)
        term3 = (n // 2)**2 * z1**m * z2**(n - 2)

        correction_poly = corr_coeff * (term1 + term2 + term3)
        beta_star_poly -= correction_poly

    # 3. Compute the intersection product
    # This is the coefficient of z1^m * z2^n * z3^(m+n-2) in the product gamma * beta_*[Tvm,n]

    # The total degree of the resulting class to integrate is dim(PS_m x PS_n x PS_{m+n-2})
    # = m + n + (m+n-2) = 2m + 2n - 2
    # The degree of gamma is m+n, degree of beta_* is m+n-2. Their sum is 2m+2n-2.

    # We need to find the coefficient of z1^m * z2^n * z3^(m+n-2) in the product
    # This is equivalent to integrating over the product of projective spaces.

    # To optimize, we only need terms from gamma_poly and beta_star_poly that can contribute.

    total_class = expand(gamma_poly * beta_star_poly)

    final_poly = Poly(total_class, z1, z2, z3)

    degree = final_poly.coeff_monomial(z1**m * z2**n * z3**(m + n - 2))

    return int(degree)

if __name__ == '__main__':
    #Interactive part for user input
    print("Enter your own values:")
    try:
        m_user = int(input("Enter m: "))
        n_user = int(input("Enter n: "))
        a_user = int(input("Enter a: "))
        b_user = int(input("Enter b: "))

        user_degree = compute_degree(m=m_user, n=n_user, a=a_user, b=b_user)
        print(f"\nComputed Degree: {user_degree}")

    except (ValueError, TypeError) as e:
        print(f"\nInvalid input or unsupported case: {e}")

\end{lstlisting}

\bibliographystyle{alpha}
 \bibliography{bibliography}

\end{document}